\documentclass{amsart}
\usepackage[a4paper,hmargin={2.5cm,2.5cm},vmargin={2.5cm,2.35cm}]{geometry}

\usepackage[utf8]{inputenc}
\usepackage[T1]{fontenc}
\usepackage{lmodern}

\usepackage[leqno]{amsmath}
\usepackage{amssymb,amsthm}
\usepackage[mathscr]{euscript}
\usepackage{bbm}
\usepackage{dsfont}
\usepackage{stmaryrd}
\usepackage[shortlabels]{enumitem}
\usepackage[all]{xy}
\usepackage{mathtools}

\makeatletter
\theoremstyle{plain}
\@ifundefined{theorem}{\newtheorem{theorem}{Theorem}[section]}{}
\@ifundefined{lemma}{\newtheorem{lemma}[theorem]{Lemma}}{}
\@ifundefined{proposition}{\newtheorem{proposition}[theorem]{Proposition}}{}
\@ifundefined{corollary}{\newtheorem{corollary}[theorem]{Corollary}}{}

\theoremstyle{definition}
\@ifundefined{definition}{\newtheorem{definition}[theorem]{Definition}}{}
\@ifundefined{examples}{}{}
\@ifundefined{example}{\newtheorem{example}[theorem]{Example}}{}
\@ifundefined{assumption}{}{}

\theoremstyle{remark}
\@ifundefined{remark}{\newtheorem{remark}[theorem]{Remark}}{}
\@ifundefined{remarks}{}{}
\makeatother

\newlist{tfae}{enumerate}{1}
\setlist[tfae,1]{label=(\roman*)}

\usepackage{chngcntr}
\counterwithin*{equation}{section}

\DeclareMathOperator{\upc}{\uparrow\!}
\DeclareMathOperator{\downc}{\downarrow\!}

\DeclareMathOperator{\im}{im}

\newcommand{\catfont}[1]{\mathsf{#1}}

\newcommand{\catA}{\catfont{A}}
\newcommand{\catB}{\catfont{B}}

\newcommand{\SET}{\catfont{Set}}

\newcommand{\ORD}{\catfont{Ord}}
\newcommand{\POST}{\catfont{Pos}}
\newcommand{\MET}{\catfont{Met}}
\newcommand{\TOP}{\catfont{Top}}
\newcommand{\STONE}{\catfont{BooSp}}
\newcommand{\PRIEST}{\catfont{Priest}}
\newcommand{\COMPHAUS}{\catfont{CompHaus}}
\newcommand{\POSCH}{\catfont{PosComp}}

\newcommand{\op}{\mathrm{op}}

\newcommand{\met}{\mathrm{met}}

\newcommand{\modto}{\mathrel{\mathmakebox[\widthof{$\xrightarrow{\rule{1.45ex}{0ex}}$}]{\xrightarrow{\rule{1.45ex}{0ex}}\hspace*{-2.8ex}{\circ}\hspace*{1ex}}}}

\DeclareMathOperator{\CoAlg}{CoAlg}
\newcommand{\luk}{\odot}

\newcommand{\monadfont}[1]{\mathbbm{#1}}
\newcommand{\mV}{\monadfont{V}}
\newcommand{\vmonad}{(V,m,e)}

\newcommand{\field}[1]{\mathds{#1}}

\newcommand{\N}{\field{N}}
\newcommand{\Q}{\field{Q}}

\newcommand{\df}[1]{\emph{\textbf{#1}}}

\title[Generating the algebraic theory of $C(X)$]{Generating the algebraic
  theory of $C(X)$:\\ the case of partially ordered compact spaces}

\author{Dirk Hofmann}

\address{Center for Research and Development in Mathematics and
Applications, Department of Mathematics, University of Aveiro,
3810-193 Aveiro, Portugal}

\email{dirk@ua.pt}

\author{Renato Neves}

\address{INESC TEC (HASLab) \& Universidade do Minho, Portugal}

\email{nevrenato@di.uminho.pt}

\author{Pedro Nora}

\address{Center for Research and Development in Mathematics and
Applications, Department of Mathematics, University of Aveiro,
3810-193 Aveiro, Portugal}

\email{a28224@ua.pt}



\date{\today}

\usepackage[backref=page,hypertexnames=false]{hyperref} 

\begin{document}

\begin{abstract}
  It is known since the late 1960's that the dual of the category of compact
  Hausdorff spaces and continuous maps is a variety -- not finitary, but bounded
  by $\aleph_1$. In this note we show that the dual of the category of partially
  ordered compact spaces and monotone continuous maps is a $\aleph_1$-ary
  quasivariety, and describe partially its algebraic theory. Based on this
  description, we extend these results to categories of Vietoris coalgebras and
  homomorphisms. We also characterise the $\aleph_1$-copresentable partially
  ordered compact spaces.
\end{abstract}

\maketitle

\section{Introduction}
\label{sec:introduction}

The motivation for this paper stems from two very different sources. Firstly, it
is known since the end of the 1960's that the dual of the category $\COMPHAUS$
of compact Hausdorff spaces and continuous maps is a variety -- not finitary,
but bounded by $\aleph_1$. More in detail,
\begin{itemize}
\item in \cite{Dus69} it is proved that the representable functor
  $\hom(-,[0,1]) \colon\COMPHAUS^\op\to\SET$ is monadic,
\item the unit interval $[0,1]$ is shown to be a $\aleph_0$-copresentable
  compact Hausdorff space in \cite{GU71},
\item a presentation of the algebra operations of $\COMPHAUS^\op$ is given in
  \cite{Isb82}, and
\item a complete description of the algebraic theory of $\COMPHAUS^\op$ is
  obtained in \cite{MR17}.
\end{itemize}
It is also worth mentioning that, by the famous Gelfand duality theorem
\cite{Gel41a}, $\COMPHAUS$ is dually equivalent to the category of commutative
$C^*$-algebras and homomorphisms; the algebraic theory of (commutative)
$C^*$-algebras is extensively studied in \cite{Neg71,PR89,PR93}. Our second
source of inspiration is the theory of coalgebras. In \cite{KKV04} the authors
argue that the category $\STONE$ of Boolean spaces and continuous maps ``is an
interesting base category for coalgebras''; among other reasons, due to the
connection with modal logic. A similar study based on the Vietoris functor on
the category $\PRIEST$ of Priestley spaces and monotone continuous maps can be
found in \cite{CLP91,Pet96,BKR07}. Arguably, the categories $\STONE$ and
$\PRIEST$ are very suitable in this context because they are duals of finitary
varieties (due to the famous Stone dualities \cite{Sto36,Sto38a,Sto38}), a
property which extends to categories of coalgebras and therefore guarantees for
instance good completeness properties.

In this note we go a step further and study the category $\POSCH$ of partially
ordered compact spaces and monotone continuous maps, which was introduced in
\cite{Nac50} and constitutes a natural extension of both the category
$\COMPHAUS$ and the category $\PRIEST$. It remains open to us whether
$\POSCH^\op$ is also a variety; however, based on the duality results of
\cite{HN16_tmp} and inspired by \cite{Isb82}, we are able to prove that
$\POSCH^\op$ is a $\aleph_1$-ary quasivariety and give a partial description of
its algebraic theory. This description turns out to be sufficient to identify
also the dual of the category of coalgebras for the Vietoris functor
$V \colon\POSCH\to\POSCH$ as a $\aleph_1$-ary quasivariety. Finally, we
characterise the $\aleph_1$-copresentable objects of $\mathsf{PosComp}$ as
precisely the metrisable ones.

\section{Preliminaries}
\label{sec:quas-posch-abstr}

In this section we recall the notion of partially ordered compact space
introduced in \cite{Nac50} together with some fundamental properties of these
spaces.

\begin{definition}
  A \df{partially ordered compact space} $(X,\leq,\tau)$ consists of a set $X$,
  a partial order $\leq$ on $X$ and a compact topology $\tau$ on $X$ so that
  \[
    \{(x,y)\in X\times X\mid x\leq y\}
  \]
  is closed in $X\times X$ with respect to the product topology.
\end{definition}

We will often simply write $X$ instead of $(X,\leq,\tau)$. For every partially
ordered compact space $X$, also the subset
\[
  \{(x,y)\in X\times X\mid x\geq y\}
\]
is closed in $X\times X$ since the mapping
$X\times X\to X\times X,\,(x,y)\mapsto (y,x)$ is a homeomorphism. Therefore the
diagonal
\[
  \Delta_X=\{(x,y)\in X\times X\mid x\leq y\}\cap\{(x,y)\in X\times X\mid x\geq
  y\}
\]
is closed in $X\times X$, which tells us that the topology of a partially
ordered compact space is Hausdorff. We denote the category of partially ordered
compact spaces and monotone continuous maps by $\POSCH$.

\begin{example}
  The unit interval $[0,1]$ with the usual Euclidean topology and the ``greater
  or equal'' relation $\geqslant$ is a partially ordered compact space; via the
  mapping $x\mapsto 1-x$, this space is isomorphic in $\POSCH$ to the space with
  the same topology and the ``less or equal'' relation $\leqslant$.
\end{example}

Clearly, there is a canonical forgetful functor $\POSCH\to\POST$ from $\POSCH$
to the category $\POST$ of partially ordered sets and monotone maps. By the
observation above, forgetting the order relation defines a functor
$\POSCH\to\COMPHAUS$ from $\POSCH$ to the category $\COMPHAUS$ of compact
Hausdorff spaces and continuous maps. For more information regarding properties
of $\POSCH$ we refer to \cite{Nac65,GHK+80,Jun04,Tho09}; however, we recall
here:

\begin{theorem}
  The category $\POSCH$ is complete and cocomplete, and both canonical forgetful
  functors $\POSCH\to\COMPHAUS$ and $\POSCH\to\POST$ preserve limits.
\end{theorem}
\begin{proof}
  This follows from the construction of limits and colimits in $\POSCH$
  described in \cite{Tho09}.
\end{proof}

We call an injective monotone continuous map $m \colon X\to Y$ between partially
ordered compact spaces an \df{embedding} in $\POSCH$ whenever $m$ is an order
embedding, that is,
\[
  x\leq y\iff m(x)\leq m(y)
\]
for all $x,y\in X$. Note that the embeddings in $\POSCH$ are, up to isomorphism,
the closed subspace inclusions with the induced order. More generally, a cone
$(f_i \colon X\to Y_i)_{i\in I}$ in $\POSCH$ is called initial whenever, for all
$x_0,x_1\in X$,
\[
  x_0\le x_1\iff \forall i\in I\,.\,f_i(x_0)\le f_i(x_1).
\]
In fact, this condition is equivalent to affirm that the cone
$(f_i \colon X\to Y_i)_{i\in I}$ is initial with respect to the forgetful
functor $\POSCH\to\COMPHAUS$ (see \cite{Tho09}). The following result of Nachbin
is crucial for our work.

\begin{theorem}\label{d:thm:2}
  The unit interval $[0,1]$ is injective in $\POSCH$ with respect to embeddings.
\end{theorem}
\begin{proof}
  See \cite[Theorem 6]{Nac65}.
\end{proof}

As we have shown in~\cite{HN16_tmp}, the theorem above has the following
important consequences.

\begin{theorem}
  The regular monomorphisms in $\POSCH$ are, up to isomorphism, the closed
  subspaces with the induced order. Consequently, $\POSCH$ has (Epi,Regular
  mono)-factorisations and the unit interval $[0,1]$ is a regular injective
  regular cogenerator of $\POSCH$.
\end{theorem}
\begin{proof}
  Based on Theorem~\ref{d:thm:2}, the characterisation of regular monomorphisms
  as precisely the embeddings can be found in \cite{HN16_tmp}, as well as a
  proof for the fact that $[0,1]$ is a regular cogenerator of $\POSCH$.
\end{proof}

We close this section with the following characterisation of cofiltered limits
in $\COMPHAUS$ which goes back to \cite[Proposition~8, page~89]{Bou66} (see also
\cite[Proposition~4.6]{Hof02} and \cite{HNN16_tmp}).

\begin{theorem}\label{thm:codirected-limits-COMPHAUS}
  Let $D \colon I\to\COMPHAUS$ be a cofiltered diagram and
  $(p_i \colon L\to D(i))_{i\in I}$ a cone for $D$. The following conditions are
  equivalent:
  \begin{tfae}
  \item The cone $(p_i \colon L\to D(i))_{i\in I}$ is a limit of $D$.
  \item The cone $(p_i \colon L\to D(i))_{i\in I}$ is mono and, for every
    $i\in I$, the image of $p_i$ is equal to the intersection of the images of
    all $D(k \colon j\to i)$ with codomain $i$:
    \[
      \im p_i=\bigcap_{j{\to}i}\im D(j\xrightarrow{\,k\,} i).
    \]
  \end{tfae}
\end{theorem}

We emphasise that this intrinsic characterisation of cofiltered limits in
$\COMPHAUS$ is formally dual to the following well-known description of filtered
colimits in $\SET$ (see \cite{AR94}).

\begin{theorem}\label{d:thm:3}
  Let $D:I\to\SET$ be a filtered diagram and $(c_i \colon D(i)\to C)_{i\in I}$ a
  compatible cocone $(c_i \colon D(i)\to C)_{i\in I}$ for $D$. The following
  conditions are equivalent:
  \begin{tfae}
  \item The cocone $(c_i \colon D(i)\to C)_{i\in I}$ is a colimit of $D$.
  \item The cocone $(c_i \colon D(i)\to C)_{i\in I}$ is epi and, for all
    $i\in I$, the coimage of $c_i$ is equal to the cointersection of the
    coimages of all $D(k \colon i\to j)$ with domain $i$:
    \[
      c_i(x)=c_i(y)\iff \exists(i\xrightarrow{\,k\,}j)\in I\,.\,D(k)(x)=D(k)(y),
    \]
    and $x,y\in D(i)$.
  \end{tfae}
\end{theorem}

\section{The quasivariety $\POSCH^\op$}
\label{sec:quas-posch-concr}

The principal aim of this section is to identify $\POSCH^\op$ as a
$\aleph_1$-ary quasivariety; moreover, we give a more concrete presentation of
the algebra structure of $\POSCH^\op$. To achieve this goal, we built on
\cite{HN16_tmp} where $\POSCH^\op$ is shown to be equivalent to the category of
certain $[0,1]$-enriched categories, for various quantale structures on the
complete lattice $[0,1]$. Arguably, the most convenient quantale structure is
the \df{\L{}ukasiewicz tensor} given by $u\luk v=\max(0,u+v-1)$, for
$u,v\in [0,1]$. For this quantale, a \df{$[0,1]$-category} is a set $X$ equipped
with a mapping $a\colon X\times X\to [0,1]$ so that
\begin{align*}
  1\leqslant a(x,x) &&\text{and}&& a(x,y)\luk a(y,z)\leqslant a(x,z),
\end{align*}
for all $x,y\in X$. Each $[0,1]$-category $(X,a)$ induces the order relation
(that is, reflexive and transitive relation)
\[
  x\leqslant y\hspace{1ex}\text{whenever}\hspace{1ex} 1\leqslant
  a(x,y)\hspace{2em}(x,y\in X)
\]
on $X$. A $[0,1]$-category is called \df{separated} whenever this order relation
is anti-symmetric. As we explain in Section~\ref{sec:aleph_1-copr-spac},
categories enriched in this quantale can be also thought of as metric spaces.

To state the duality result of \cite{HN16_tmp}, we need to impose certain
(co)completeness conditions on $[0,1]$-categories. Since these conditions will
not be used explicitly in this paper, we simply refer to \cite{HN16_tmp} for
their definitions. Eventually, we consider the category $\catA$ with objects all
separated finitely cocomplete $[0,1]$-categories with a monoid structure that,
moreover, admit $[0,1]$-powers; the morphisms of $\catA$ are the finitely
cocontinuous $[0,1]$-functors preserving the monoid structure and the
$[0,1]$-powers. Alternatively, these structures can be described algebraically
as sup-semilattices with actions of $[0,1]$; therefore we simply refer to
\cite{Kel82,Stu14} for information about enriched categories and proceed by
describing $\catA$ as a category of algebras.

\begin{remark}\label{d:rem:1}
  The category $\catA$ together with its canonical forgetful functor
  $\catA\to\SET$ is a $\aleph_1$-ary quasivariety; we recall now the
  presentation given in \cite{HN16_tmp}. For more information on varieties and
  quasivarieties we refer to \cite{AR94}. Firstly, the set of operation symbols
  consists of
  \begin{itemize}
  \item the nullary operation symbols $\bot$ and $\top$;
  \item the unary operation symbols $-\luk u$ and $-\pitchfork u$, for each
    $u\in [0,1]$;
  \item the binary operation symbols $\vee$ and $\ocircle$.
  \end{itemize}
  Secondly, the algebras for this theory should be sup-semilattices with a
  supremum-preserving action of $[0,1]$; writing $x\le y$ as an abbreviation for
  the equation $y=x\vee y$, this translates to the equations and implications
  \begin{align*}
    x\vee x & =x,&  x\vee (y\vee z)&=(x\vee y)\vee z,& x\vee\bot&=x, & x\vee y&=y\vee x,\\
    x\luk 1&=x,& (x\luk u)\luk v&=x\luk (u\luk v),& \bot\luk
                                                    u&=\bot,& (x\vee y)\luk u &=(x\luk u)\vee(y\luk u),
  \end{align*}
  \[
    x\luk u\leq x\luk v \hspace{1em}\text{ and }\hspace{1em} \bigwedge_{u\in
      S}(x\luk u\le y)\Longrightarrow (x\luk v\le y) \hspace{1em}(S\subseteq
    [0,1]\text{ countable and } v=\sup S).
  \]
  
  The algebras defined by the operations $\bot$, $\vee$ and $-\luk u$
  ($u\in [0,1]$) and the equations above are precisely the separated
  $[0,1]$-categories with finite weighted colimits. Such a $[0,1]$-category
  $(X,a)$ has all \emph{powers} $x\pitchfork u$ ($x\in X,u\in [0,1]$) if and
  only if, for all $u\in [0,1]$, $-\luk u$ has a right adjoint $-\pitchfork u$
  with respect to the underlying order. Therefore we add to our theory the
  implications
  \[
    x\luk u\le y \iff x\le y\pitchfork u,
  \]
  for all $u\in [0,1]$. Finally, regarding $\ocircle$, we impose the commutative
  monoid axioms with neutral element the top-element:
  \begin{align*}
    x\ocircle y&=y\ocircle x, & x\ocircle(y\ocircle z)&=(x\ocircle y)\ocircle z, & x\ocircle \top& =x, & \top &\le x.
  \end{align*}
  Moreover, we require this multiplication to preserve suprema and the action
  $-\odot u$ (for $u\in [0,1]$) in each variable:
  \begin{align*}
    x\ocircle(y\vee z) &= (x\ocircle y)\vee(x\ocircle z),
    & x\ocircle \bot &= \bot,
    & x\ocircle(y\odot u) &= (x\ocircle y)\odot u.
  \end{align*}
\end{remark}
\begin{remark}
  The unit interval $[0,1]$ becomes an algebra for the theory above with
  $\ocircle=\odot$ and $v\pitchfork u=\min(1,1-u+v)=1-\max(0,u-v)$, and the
  usual interpretation of all other symbols.
\end{remark}

The following result is in \cite{HN16_tmp}.

\begin{theorem}
  The functor
  \[
    C\colon \POSCH^\op\longrightarrow\catA
  \]
  sending $f\colon X\to Y$ to $Cf\colon CY\to CX,\,\psi\mapsto\psi\cdot f$ is
  fully faithful, here the structure on
  \[
    CX=\{f\colon X\to[0,1]\mid f\text{ is monotone and continuous}\}
  \]
  is defined pointwise.
\end{theorem}

\begin{remark}\label{d:rem:2}
  The theorem above remains valid if we augment the algebraic theory of $\catA$
  by further operation symbols corresponding to monotone continuous functions
  $[0,1]^I\to[0,1]$. More precisely, let $\aleph$ be a cardinal and
  $h \colon[0,1]^\aleph\to[0,1]$ be a monotone continuous map. If we add to the
  algebraic theory of $\catA$ a operation symbol of arity $\aleph$, then
  $C\colon\POSCH^\op\to\catA$ lifts to a fully faithful functor from
  $\POSCH^\op$ to the category of algebras for this theory by interpreting the
  new operation symbol in $CX$ point-wise by $h$. Note that all
  $\catA$-morphisms of type $CY\to CX$ preserves this new operation
  automatically.
\end{remark}

\begin{remark}
  Note that $1-u=0\pitchfork u$, for every $u\in [0,1]$. Therefore we can
  express truncated minus $v\ominus u=\max(0,v-u)$ in $[0,1]$ with the
  operations of $\catA$:
  \[
    v\ominus u=0\pitchfork(u\pitchfork v).
  \]
  In particular, every subalgebra $M\subseteq CX$ of $CX$ is also closed under
  truncated minus.
\end{remark}

Since we have chosen the \L{}ukasiewicz tensor, the categorical closure on $CX$
(see \cite{HT10}) coincides with the usual topology induced by the
``$\sup$-metric'' on $CX$; in the sequel we consider this topology on
$C(X)$. One important step towards the identification of the image of
$C\colon \POSCH^\op\longrightarrow\catA$ is the following adaption of the
classical ``Stone--Weierstra\ss{} theorem'' (see \cite{HN16_tmp}).

\begin{theorem}\label{d:thm:1}
  Let $X$ be a partially ordered compact space and $m\colon A\hookrightarrow CX$
  be a subobject of $CX$ in $\catA$ so that the cone $(m(a)\colon X\to[0,1])$ is
  point-separating and initial. Then $m$ is dense in $CX$. In particular, if $A$
  is Cauchy complete, then $m$ is an isomorphism.
\end{theorem}

One important consequence of Theorem~\ref{d:thm:1} is the following proposition.

\begin{proposition}\label{prop:unit-interval-copresentable}
  The unit interval $[0,1]$ is $\aleph_1$-copresentable in $\POSCH$.
\end{proposition}
\begin{proof}
  This can be shown with the same argument as in \cite[6.5.(c)]{GU71}. Firstly,
  by Theorem~\ref{d:thm:1}, $\hom(-,[0,1])$ sends every $\aleph_1$-codirected
  limit to a jointly surjective cocone. Secondly, using Theorem~\ref{d:thm:3},
  this cocone is a colimit since $[0,1]$ is $\aleph_1$-copresentable in
  $\COMPHAUS$.
\end{proof}

\begin{theorem}
  The functor $C\colon \POSCH^\op\to\catA$ corestricts to an equivalence between
  $\POSCH^\op$ and the full subcategory of $\catA$ defined by those objects $A$
  which are Cauchy complete and where the cone of all $\catA$-morphisms from $A$
  to $[0,1]$ is point-separating.
\end{theorem}
\begin{proof}
  See \cite{HN16_tmp}.
\end{proof}

Instead of working with Cauchy completeness, we wish to add an operation to the
algebraic theory of $\catA$ so that, if $M$ is closed in $CX$ under this
operation, then $M$ is closed with respect to the topology of the
$[0,1]$-category $CX$. In the case of $\COMPHAUS$, this is achieved in
\cite{Isb82} using the operation
\[
  [0,1]^\N \longrightarrow [0,1],\,
  (u_n)_{n\in\N}\longmapsto\sum_{n=0}^\infty\frac{1}{2^{n+1}}u_n
\]
on $[0,1]$; since the limit of a convergent sequence $(\varphi_n)_{n\in\N}$ can
be calculated as
\[
  \lim_{n\to\infty}\varphi_n=\varphi_0+(\varphi_1-\varphi_0)+\dots.
\]
However, this argument cannot be transported directly into the ordered setting
since the difference $\varphi_1-\varphi_0$ of two monotone maps
$\varphi_0,\varphi_1\colon X\to [0,1]$ is not necessarily monotone. To
circumvent this problem, we look for a monotone and continuous function
$[0,1]^\N\to[0,1]$ which calculates the limit of ``sufficiently many
sequences''. We make now the meaning of ``sufficiently many'' more precise.

\begin{lemma}
  Let $M\subseteq CX$ be a subalgebra in $\catA$ and $\psi\in CX$ with
  $\psi\in\overline{M}$. Then there exists a sequence $(\psi_n)_{n\in\N}$ in $M$
  converging to $\psi$ so that
  \begin{enumerate}
  \item $(\psi_n)_{n\in\N}$ is increasing, and
  \item for all $n\in\N$ and all $x\in X$:
    $\psi_{n+1}(x)-\psi_n(x)\le\frac{1}{2^n}$.
  \end{enumerate}
\end{lemma}
\begin{proof}
  We can find $(\psi_n)_{n\in\N}$ so that, for all $n\in\N$,
  $|\psi_n(x)-\psi(x)|\le\frac{1}{n+1}$. Then the sequence
  $(\psi_n\ominus\frac{1}{n+1})_{n\in\N}$ converges to $\psi$ too; moreover,
  since $M\subseteq CX$ is a subalgebra, also $\psi_n\ominus\frac{1}{n+1}\in M$,
  for all $n\in\N$. Therefore we can assume that we have a sequence
  $(\psi_n)_{n\in\N}$ in $M$ with $(\psi_n)_{n\in\N}\to\psi$ and
  $\psi_n\le\psi$, for all $n\in\N$. Then the sequence
  $(\psi_0\vee\dots\vee\psi_n)_{n\in\N}$ has all its members in $M$, is
  increasing and converges to $\psi$. Finally, there is a subsequence of this
  sequence which satisfies the second condition above.
\end{proof}

\begin{lemma}
  Every increasing sequence $(u_n)_{n\in\N}$ in $[0,1]$ satisfying
  $u_{n+1}-u_n\le\frac{1}{2^n}$, for all $n\in\N$, is Cauchy. Let
  \[
    \mathcal{C}=\{(u_n)_{n\in\N}\in[0,1]^\N\mid(u_n)_{n\in\N}\text{ is monotone
      and $u_{n+1}-u_n\le\frac{1}{2^n}$, for all $n\in\N$}\}.
  \]
  Then every sequence in $\mathcal{C}$ is Cauchy and
  $\lim\colon\mathcal{C}\to[0,1]$ is monotone and continuous.
\end{lemma}

Motivated by the two lemmas above, we are looking for a monotone continuous map
$[0,1]^\N\to[0,1]$ which sends every sequence in $\mathcal{C}$ to its
limit. Such a map can be obtained by combining $\lim\colon\mathcal{C}\to[0,1]$
with a monotone continuous retraction of the inclusion map
$\mathcal{C}\hookrightarrow [0,1]^\N$.

\begin{lemma}
  The map
  $\mu\colon [0,1]^\N\to[0,1]^\N,\,(u_n)_{n\in\N}\mapsto(u_0\vee\dots\vee
  u_n)_{n\in\N}$ is monotone and continuous.
\end{lemma}
Clearly, $\mu$ sends a sequence to an increasing sequence, and
$\mu((u_n)_{n\in\N})=(u_n)_{n\in\N}$ for every increasing sequence
$(u_n)_{n\in\N}$.

\begin{lemma}
  The map $\gamma\colon [0,1]^\N\to[0,1]^\N$ sending a sequence $(u_n)_{n\in\N}$
  to the sequence $(v_n)_{n\in\N}$ defined recursively by
  \begin{align*}
    v_0=u_0 &&\text{and}&& v_{n+1}=\min\left(u_{n+1},v_n+\frac{1}{2^n}\right)
  \end{align*}
  is monotone and continuous. Furthermore, $\gamma$ sends an increasing sequence
  to an increasing sequence.
\end{lemma}
\begin{proof}
  It is easy to see that $\gamma$ is monotone. In order to verify continuity, we
  consider $\N$ as a discrete topological space, this way $[0,1]^\N$ is an
  exponential in $\TOP$. We show that $\gamma$ corresponds via the exponential
  law to a (necessarily continuous) map
  $f\colon \N\to[0,1]^{\left([0,1]^\N\right)}$. The recursion data above
  translates to the conditions
  \begin{align*}
    f(0)=\pi_0 &&\text{and}&& f(n+1)((u_m)_{m\in\N})=\min\left(u_{n+1},f(n)((u_m)_{m\in\N})+\frac{1}{2^n}\right),
  \end{align*}
  that is, $f$ is defined by the recursion data
  $\pi_0\in[0,1]^{\left([0,1]^\N\right)}$ and
  \[
    [0,1]^{\left([0,1]^\N\right)} \longrightarrow
    [0,1]^{\left([0,1]^\N\right)},\,\varphi \longmapsto \min
    \left(\pi_{n+1},\varphi+\frac{1}{2^n}\right).
  \]
  Note that with $\varphi\colon [0,1]^\N\to[0,1]$ also
  $\min\left(\pi_{n+1},\varphi+\frac{1}{2^n}\right)\colon [0,1]^\N\to[0,1]$ is
  continuous. Finally, if $(u_n)_{n\in\N}$ is increasing, then so is
  $(v_n)_{n\in\N}$.
\end{proof}

We conclude that the map $\gamma\cdot\mu\colon [0,1]^\N\to\mathcal{C}$ is a
retraction for the inclusion map $\mathcal{C}\to[0,1]^\N$ in $\POSCH$. Therefore
we define now:
\begin{definition}
  Let $\overline{\catA}$ be the $\aleph_1$-ary quasivariety obtained by adding
  one $\aleph_1$-ary operation symbol to the theory of $\catA$ (see
  Remark~\ref{d:rem:1}). Then $[0,1]$ becomes an object of $\overline{\catA}$ by
  interpreting this operation symbol by
  \[
    \delta=\lim\cdot\gamma\cdot\mu\colon [0,1]^\N\to[0,1].
  \]
\end{definition}
The (accordingly modified) functor $C\colon \POSCH\to\overline{\catA}$ is fully
faithful (see Remark~\ref{d:rem:2}); moreover, by
Proposition~\ref{prop:unit-interval-copresentable}, $C$ sends
$\aleph_1$-codirected limits to $\aleph_1$-directed colimits in
$\overline{\catA}$.
\begin{definition}
  Let $\overline{\catA}_0$ be the subcategory of $\overline{\catA}$ defined by
  those objects $A$ where the cone of all morphisms from $A$ to $[0,1]$ is
  point-separating.
\end{definition}
Hence, $\overline{\catA}_0$ is a regular epireflective full subcategory of
$\overline{\catA}$ and therefore also a quasivariety. Moreover:
\begin{theorem}
  The embedding $C\colon \POSCH^\op\to\overline{\catA}$ corestricts to an
  equivalence $C\colon \POSCH^\op\to\overline{\catA}_0$. Hence,
  $\overline{\catA}_0$ is closed in $\overline{\catA}$ under $\aleph_1$-directed
  colimits and therefore also a $\aleph_1$-ary quasivariety (see \cite[Remark
  3.32]{AR94}).
\end{theorem}

\section{Vietoris coalgebras}
\label{sec:vietoris-coalgebras}

In this section we consider the Vietoris functor $V\colon \POSCH\to\POSCH$ and
the associated category $\CoAlg(V)$ of coalgebras and homomorphisms. We show
that $\CoAlg(V)$ as well as certain full subcategories are also $\aleph_1$-ary
quasivarieties.

Recall from \cite{Sch93} (see also \cite[Proposition~3.28]{HN16_tmp}) that, for
a partially ordered compact space $X$, the elements of $VX$ are the closed upper
subsets of $X$, the order on $VX$ is containment $\supseteq$, and the sets
\begin{align*}
  \{A\in VX\mid A\cap U\neq\varnothing\}\quad\text{($U\subseteq
  X$ open lower)}
  &&\text{and}
  && \{A\in VX\mid A\cap
     K=\varnothing\}\quad\text{($K\subseteq
     X$ closed lower)}
\end{align*}
generate the compact Hausdorff topology on $VX$. Furthermore, for
$f\colon X\to Y$ in $\POSCH$, the map $Vf\colon VX\to VY$ sends $A$ to the
up-closure $\upc f[A]$ of $f[A]$. A coalgebra $(X,\alpha)$ for $V$ consists of a
partially ordered compact space $X$ and a monotone continuous map
$\alpha \colon X\to VX$. For coalgebras $(X,\alpha)$ and $(Y,\beta)$, a
homomorphism of coalgebras $f \colon(X,\alpha)\to(Y,\beta)$ is a monotone
continuous map $f \colon X\to Y$ so that the diagram
\[
  \xymatrix{X\ar[r]^{f}\ar[d]_{\alpha} & Y\ar[d]^{\beta} \\ VX\ar[r]_{Vf} & VY}
\]
commutes. The coalgebras for the Vietoris functor and their homomorphisms form
the category $\CoAlg(V)$, and forgetting the coalgebra structure gives rise to
the canonical forgetful functor $\CoAlg(V)\to \POSCH$ that sends $(X,\alpha)$ to
$X$ and leaves the maps unchanged. For the general theory of coalgebras we refer
to \cite{Ada05}.

As it is well-known, $V$ is part of a monad $\mV=\vmonad$ on $\POSCH$; here
$e_X \colon X\to VX$ sends $x$ to $\upc x$ and $m_X \colon VVX\to VX$ is given
by $\mathcal{A}\mapsto\bigcup\mathcal{A}$. Clearly, a coalgebra structure
$X\to VX$ for $V$ can be also interpreted as an endomorphism $X\modto X$ in the
Kleisli category $\POSCH_\mV$. In the sequel we will use this perspective
together with the duality result for $\POSCH_\mV$ of \cite{HN16_tmp} to show
that also $\CoAlg(V)^\op$ is a $\aleph_0$-ary quasivariety.

Let $\overline{\catB}$ denote the category with the same objects as
$\overline{A}$ and morphisms those maps $\varphi \colon A\to A'$ that preserve
finite suprema and the action $-\luk u$, for all $u\in [0,1]$, and satisfy
\[
  \varphi(x\ocircle y)\le \varphi(x)\ocircle \varphi(y),
\]
for all $x,y\in A$. The functor $C\colon \POSCH^\op\to\overline{\catA}$ extends
to a fully faithful functor $C\colon \POSCH_\mV\to\overline{\catB}$ making the
diagram
\[
  \xymatrix{\POSCH_\mV^\op\ar[r]^-{C} & \overline{\catB}\\
    \POSCH^\op\ar[r]_-{C}\ar[u] & \overline{\catA}_0\ar[u]}
\]
commutative, where the vertical arrows denote the canonical inclusion functors.
Therefore the category $\CoAlg(V)$ is dually equivalent to the category with
objects all pairs $(A,a)$ consisting of an $\overline{\catA}_0$ object $A$ and a
$\overline{\catB}$-morphism $a\colon A\to A$, and a morphism between such pairs
$(A,a)$ and $(A',a')$ is an $\overline{\catA}_0$-morphism $A\to A'$ commuting in
the obvious sense with $a$ and $a'$.

\begin{theorem}
  The category $\CoAlg(V)$ of coalgebras and homomorphisms for the Vietoris
  functor $V\colon \POSCH\to\POSCH$ is dually equivalent to a $\aleph_1$-ary
  quasivariety.
\end{theorem}
\begin{proof}
  Just consider the algebraic theory of $\overline{\catA}_0$ augmented by one
  unary operation symbol and by those equations which express that the
  corresponding operation is a $\overline{\catB}$-morphism.
\end{proof}

In particular, $\CoAlg(V)$ is complete and the forgetful functor
$\CoAlg(V)\to\POSCH$ preserves $\aleph_1$-codirected limits. In fact, slightly
more is shown in \cite{HNN16_tmp}:
\begin{proposition}
  The forgetful functor $\CoAlg(V)\to\POSCH$ preserves codirected limits.
\end{proposition}

We finish this section by pointing out some further consequences of our approach
and consider certain full subcategories of $\CoAlg(V)$. For instance, still
thinking of a coalgebra structure $\alpha \colon X\to VX$ as an endomorphism
$\alpha \colon X\modto X$ in $\POSCH_\mV$, we say that $\alpha$ is
\df{reflexive} whenever $1_X\le\alpha$ in $\POSCH_\mV$, and $\alpha$ is called
\df{transitive} whenever $\alpha\circ\alpha\le\alpha$ in $\POSCH_\mV$. Passing
now to the corresponding $\overline{B}$-morphism $a \colon A\to A$, these
inequalities can be expressed as equations in $A$, and we conclude:

\begin{proposition}
  The full subcategory of $\CoAlg(V)$ defined by all reflexive (transitive,
  reflexive and transitive) coalgebras is dually equivalent to an $\aleph_1$-ary
  quasivariety. Moreover, this subcategory is coreflective in $\CoAlg(V)$ and
  closed under $\aleph_1$-ary limits.
\end{proposition}
\begin{proof}
  This follows from the discussion above and from \cite[Theorem~1.66]{AR94}.
\end{proof}

Another way of specifying full subcategories of $\CoAlg(V)$ uses coequations
(see \cite[Definition~4.18]{Ada05}). More generally, for a class $\mathcal{M}$
of monomorphisms in $\CoAlg(V)$, a coalgebra $X$ for $V$ is called
\df{coorthogonal} whenever, for all $m \colon A\to B$ in $\mathcal{M}$ and all
homomorphisms $f \colon X\to B$ there exists a (necessarily unique) homomorphism
$g \colon X\to A$ with $m\cdot g=f$ (see \cite[Definition~1.32]{AR94} for the
dual notion). We write $\mathcal{M}^\top$ for the full subcategory of
$\CoAlg(V)$ defined by those coalgebras which are coorthogonal to
$\mathcal{M}$. From the dual of \cite[Theorem~1.39]{AR94} we obtain:

\begin{proposition}
  For every set $\mathcal{M}$ of monomorphisms in $\CoAlg(V)$, the inclusion
  functor $\mathcal{M}^\top\hookrightarrow\CoAlg(V)$ has a right
  adjoint. Moreover, if $\lambda$ denotes a regular cardinal larger or equal to
  $\aleph_1$ so that, for every arrow $m\in\mathcal{M}$, the domain and codomain
  of $m$ is $\lambda$-copresentable, then
  $\mathcal{M}^\top\hookrightarrow\CoAlg(V)$ is closed under
  $\lambda$-codirected limits.
\end{proposition}

\begin{corollary}
  For every set of coequations in $\CoAlg(V)$, the full subcategory of
  $\CoAlg(V)$ defined by these coequations is coreflective.
\end{corollary}

\section{$\aleph_1$-copresentable spaces}
\label{sec:aleph_1-copr-spac}

It is shown in \cite{GU71} that the the $\aleph_1$-copresentable objects in
$\COMPHAUS$ are precisely the metrisable compact Hausdorff spaces. We end this
paper with a characterisation of the $\aleph_1$-copresentable objects in
$\POSCH$ which resembles the one for compact Hausdorff spaces; to do so, we
consider generalised metric spaces in the sense of Lawvere \cite{Law73}.

More precisely, we think of metric spaces as categories enriched in the quantale
$[0,1]$, ordered by the ``greater or equal'' relation $\geqslant$, with tensor
product $\oplus$ given by truncated addition:
\[
  u\oplus v=\min(1,u+v),
\]
for all $u,v\in [0,1]$. We note that the right adjoint $\hom(u,-)$ of
$u\oplus-\colon [0,1]\to [0,1]$ is defined by
\[
  \hom(u,v)=v\ominus u=\max(0,v-u),
\]
for all $u,v\in [0,1]$.
\begin{remark}
  Via the isomorhism $[0,1]\to [0,1],\,u\mapsto 1-u$, the quantale described
  above is isomorphic to the quantale $[0,1]$ equipped with the \L{}ukasiewicz
  tensor used in Section~\ref{sec:quas-posch-concr}. However, we decided to
  switch so that categories enriched in $[0,1]$ look more like metric spaces.
\end{remark}

\begin{definition}
  A \df{metric space} is a pair $(X,a)$ consisting of a set $X$ and a map
  $a\colon X\times X\to [0,1]$ satisfying
  \begin{align*}
    0\geqslant a(x,x) &&\text{and}&& a(x,y)\oplus a(y,z)\geqslant a(x,z),
  \end{align*}
  for all $x,y,z\in X$. A map $f\colon X\to Y$ between metric spaces $(X,a)$,
  $(Y,b)$ is called \df{non-expansive} whenever
  \[
    a(x,x')\geqslant b(f(x),f(x')),
  \]
  for all $x,x'\in X$. Metric spaces and non-expansive maps form the category
  $\MET$.
\end{definition}

\begin{example}\label{d:ex:1}
  The unit interval $[0,1]$ is a metric space with metric
  $\hom(u,v)=v\ominus u$.
\end{example}

Our definition of metric space is not the classical one. Firstly, we consider
only metrics bounded by $1$; however, since we are interested in the induced
topology and the induced order, ``large'' distances are irrelevant. Secondly, we
allow distance zero for different points, which, besides topology, also allows
us to treat order theory. Every metric $a$ on a set $X$ defines the order
relation
\[
  x\le y \hspace{1ex}\text{whenever}\hspace{1ex} 0\geqslant a(x,y),
\]
for all $x,y\in X$; this construction defines a functor
\[
  O\colon \MET \longrightarrow\ORD
\]
commuting with the canonical forgetful functors to $\SET$.

\begin{lemma}
  The functor $ O\colon \MET\to\ORD$ preserves limits and initial cones.
\end{lemma}

A metric space is called \df{separated} whenever the underlying order is
anti-symmetric. Element-wise, a metric space $(X,a)$ is separated whenever
\[
  (0\geqslant a(x,y) \hspace{1ex}\&\hspace{1ex} 0\geqslant a(y,z)) \implies x=y,
\]
for all $x,y\in X$.

Thirdly, we are not insisting on symmetry. However, every metric space $(X,a)$
can be symmetrised by
\[
  a_s(x,y)=\max(a(x,y),a(y,x)).
\]
For every metric space $(X,a)$, we consider the topology induced by the
symmetrisation $a_s$ of $a$. This construction defines the faithful functor
\[
  T\colon \MET \longrightarrow\TOP.
\]
We note that $(X,a)$ is separated if and only if the underlying topology is
Hausdorff. Furthermore, we recall:

\begin{lemma}\label{d:lem:1}
  The functor $T\colon \MET\to\TOP$ preserves finite limits. In particular, $T$
  sends subspace embeddings to subspace embeddings.
\end{lemma}

\begin{lemma}
  Let $(X,a)$ be a separated compact metric space. Then $X$ equipped with the
  order and the topology induced by the metric $a$ becomes a partially ordered
  compact space.
\end{lemma}
\begin{proof}
  See \cite[Chapter~II]{Nac65}.
\end{proof}

\begin{example}
  The metric space $[0,1]$ of Example~\ref{d:ex:1} induces the partially ordered
  compact Hausdorff space $[0,1]$ with the usual Euclidean topology and the
  ``greater or equal'' relation $\geqslant$.
\end{example}

\begin{definition}
  A partially ordered compact space $X$ is called \df{metrisable} whenever there
  is a metric on $X$ which induces the order and the topology of $X$. We denote
  the full subcategory of $\POSCH$ defined by all metrisable spaces by
  $\POSCH_\met$.
\end{definition}

\begin{proposition}\label{d:prop:1}
  $\POSCH_\met$ is closed under countable limits in $\POSCH$.
\end{proposition}
\begin{proof}
  By Lemma~\ref{d:lem:1}, $\POSCH_\met$is closed under finite limits in
  $\POSCH$. The argument for countable products is the same as in the classical
  case: for a family $(X_n)_{n\in\N}$ of metrisable partially ordered compact
  Hausdorff spaces, with the metric $a_n$ on $X_n$ ($n\in\N$), the structure of
  the product space $X=\prod_{n\in\N}X_n$ is induced by the metric $a$ defined
  by
  \[
    a((x_n)_{n\in\N},(y_n)_{n\in\N})=\sum_{n=0}^\infty\frac{1}{2^{n+1}}a_n(x_n,y_n),
  \]
  for $(x_n)_{n\in\N},(y_n)_{n\in\N}\in X$.
\end{proof}

For classical metric spaces it is known that the compact spaces are subspaces of
countable powers of the unit interval; this fact carries over without any
trouble to our case. Before we present the argument, let us recall that, for
every metric space $(X,a)$, the cone $(a(x,-)\colon X\to [0,1])_{x\in X}$ is
initial with respect to the forgetful functor $\MET\to\SET$; this is a
consequence of the Yoneda Lemma for $[0,1]_\oplus$-categories (see
\cite{Law73}). Moreover, $(X,a)$ is separated if and only if this cone is
point-separating.

\begin{lemma}\label{d:lem:2}
  Let $(X,a)$ be a compact metric space. Then there exists a countable subset
  $S\subseteq X$ so that the cone
  \[
    (a(z,-)\colon X\to [0,1])_{z\in S}
  \]
  is initial with respect to the forgetful functor $\MET\to\SET$.
\end{lemma}
\begin{proof}
  Since $X$ is compact, for every natural number $n\geq 1$, there is a finite
  set $S_n$ so that the open balls
  \[
    \{y\in X\mid a(x,y)<\frac{1}{n}\text{ and }a(y,x)<\frac{1}{n}\}
  \]
  with $x\in S_n$ cover $X$. Let $S=\bigcup_{n\geq 1}S_n$. We have to show that,
  for all $x,y\in X$,
  \[
    \bigvee_{z\in S} a(z,y)\ominus a(z,x)\geqslant a(x,y).
  \]
  To see that, let $\varepsilon=\frac{1}{n}$, for some $n\geq 1$. By
  construction, there is some $z\in S$ so that $a(x,z)<\varepsilon$ and
  $a(z,x)<\varepsilon$. Hence,
  \[
    (a(z,y)\ominus a(z,x))+2\varepsilon \geqslant a(z,y)+a(x,z)\geqslant a(x,y);
  \]
  and the assertion follows.
\end{proof}

\begin{proposition}\label{d:prop:2}
  Every partially ordered compact space is a $\aleph_1$-cofiltered limit in
  $\POSCH$ of metrisable spaces.
\end{proposition}
\begin{proof}
  For a separated metric space $X=(X,a)$, the initial cone
  $(a(x,-)\colon X\to [0,1])_{x\in S}$ of Lemma~\ref{d:lem:2} is automatically
  point-separating, therefore there is an embedding $X\hookrightarrow [0,1]^\N$
  in $\MET$. This proofs that the full subcategory $\POSCH_\met$ of $\POSCH$ is
  small. Let $X$ be a partially ordered compact space. By
  Proposition~\ref{d:prop:1}, the canonical diagram
  \[
    D\colon X\downarrow \POSCH_\met \longrightarrow \POSCH
  \]
  is $\aleph_1$-cofiltered. Moreover, the canonical cone
  \begin{equation}\label{d:eq:2}
    (f\colon X\to Y)_{f\in (X\downarrow \POSCH_\met)}
  \end{equation}
  is initial since \eqref{d:eq:2} includes the cone $(f\colon X\to
  [0,1])_f$. Finally, to see that \eqref{d:eq:2} is a limit cone, we use
  Theorem~\ref{thm:codirected-limits-COMPHAUS}: for every $f\colon X\to Y$ with
  $Y$ metrisable, $\im f\hookrightarrow Y$ actually belongs to $\POSCH_\met$,
  which proves
  \[
    \im f=\bigcap_{k\colon g\to f\in(X\downarrow \POSCH_\met)}\im D(k).\qedhere
  \]
\end{proof}

\begin{corollary}
  Every $\aleph_1$-copresentable object in $\POSCH$ is metrisable.
\end{corollary}
\begin{proof}
  Also here the argument is the same as for $\COMPHAUS$. Let $X$ be a
  $\aleph_1$-copresentable object in $\POSCH$. By Proposition~\ref{d:prop:2}, we
  can present $X$ as a limit $(p_i\colon X\to X_i)_{i\in I}$ of a
  $\aleph_1$-codirected diagram $D\colon I\to\POSCH$ where all $D(i)$ are
  metrisable. Since $X$ is $\aleph_1$-copresentable, the identity
  $1_X\colon X\to X$ factorises as
  \[
    X\xrightarrow{\;p_{i}\;} X_i\xrightarrow{\;h\;}X,
  \]
  for some $i\in I$. Hence, being a subspace of a metrisable space, $X$ is
  metrisable.
\end{proof}

To prove that every metrisable partially ordered compact space $X$ is
$\aleph_1$-copresentable, we will show that every closed subspace
$A\hookrightarrow [0,1]^I$ with $I$ countable is an equaliser of a pair of
arrows
\[
  \xymatrix{[0,1]^I\ar@<-.5ex>[r] \ar@<.5ex>[r] & [0,1]^J}
\]
with also $J$ being countable. For a symmetric metric on $X$, one can simply
consider
\[
  \xymatrix{[0,1]^I\ar@<-.5ex>[r]_{0} \ar@<.5ex>[r]^{a(A,-)} & [0,1],}
\]
but in our non-symmetric setting this argument does not work. We start with an
auxiliary result which follows from Theorem~\ref{d:thm:2}. Our argument here is
a slight modification of the one used in the characterisation of regular
monomorphisms in $\POSCH$ obtained in \cite{HN16_tmp}.

\begin{lemma}\label{d:lem:3}
  Let $X$ be a partially ordered compact space, $A,B\subseteq X$ closed subsets
  so that $A\cap B=\varnothing$ and $B=\downc B\cap \upc B$. Then there is a
  family $(f_u\colon X\to [0,1])_{u\in [0,1]}$ of monotone continuous maps which
  all coincide on $A$ and, moreover, satisfy $f_u(y)=u$, for all $u\in [0,1]$
  and $y\in B$.
\end{lemma}
\begin{proof}
  Put $A_0=A\cap \upc B$ and $A_1=A\cap\downc B$. Then $A_0$ and $A_1$ are
  closed subsets of $X$ and
  \[
    A_0\cap A_1=A\cap\upc B\cap\downc B=A\cap B=\varnothing.
  \]
  Moreover, for every $x_0\in A_0$ and $x_1\in A_1$, $x_0\nleq x_1$. In fact, if
  $x_0\geq y_0\in B$ and $x_1\leq y_1\in B$, then $x_0\leq x_1$ implies
  \[
    y_0\leq x_0\leq x_1\leq y_1,
  \]
  hence $x_0\in B$ which contradicts $A\cap B=\varnothing$. We define now the
  monotone continuous map
  \begin{align*}
    g \colon A_0\cup A_1 &\longrightarrow [0,1]\\
    x & \longmapsto
        \begin{cases}
          0 & \text{if }x\in A_0,\\
          1 & \text{if }x\in A_1.
        \end{cases}
  \end{align*}
  By \cite[Theorem 6]{Nac65}, $g$ extends to a monotone continuous map
  $g\colon A\to [0,1]$. Let now $u\in [0,1]$. We define
  \begin{align*}
    f_u \colon A\cup B & \longrightarrow [0,1]\\
    x & \longmapsto
        \begin{cases}
          g(x) & \text{if }x\in A,\\
          u & \text{if }x\in B.
        \end{cases}
  \end{align*}
  Using again \cite[Theorem 6]{Nac65}, $f_u$ extends to a monotone continuous
  map $f_u \colon X\to [0,1]$.
\end{proof}

\begin{corollary}\label{d:cor:1}
  Let $n\in \N$ and $A\subseteq [0,1]^n$ be a closed subset. Then there exist
  countable set $J$ and monotone continuous maps
  \[
    \xymatrix{[0,1]^n\ar@<-.5ex>[r]_{k} \ar@<.5ex>[r]^{h} & [0,1]^J}
  \]
  so that $A\hookrightarrow [0,1]^n$ is the equaliser of $h$ and $k$. In
  particular, $A$ is $\aleph_1$-copresentable.
\end{corollary}
\begin{proof}
  We denote by $d$ the usual Euclidean metric on $[0,1]^n$. For every
  $x\in [0,1]^n$ with $x\notin A$, there is some $\varepsilon>0$ so that the
  closed ball
  \[
    B(x,\varepsilon)=\{y\in [0,1]^n\mid d(x,y)\leqslant\varepsilon\}
  \]
  does not intersect $A$. Furthermore, $B=\upc B\cap\downc B$. Put
  \[
    J=\{(k,x_1,\dots x_n)\mid k\in\N,k\geq 1 \text{ and } x=(x_1,\dots x_n)\in
    ([0,1]\cap\Q)^n \text{ and } B(x,\frac{1}{k})\cap A=\varnothing\};
  \]
  clearly, $J$ is countable. For every $j=(k,x_1,\dots x_n)\in J$, we consider
  the monotone continuous maps $f_0,f_1 \colon [0,1]^n\to [0,1]$ obtained in
  Lemma~\ref{d:lem:3} and put $h_j=f_0$ and $k_j=f_1$. Then
  $A\hookrightarrow [0,1]^n$ is the equaliser of
  \[
    \xymatrix{[0,1]^n \ar@<-.5ex>[r]_{k=\langle k_j\rangle}
      \ar@<.5ex>[r]^{h=\langle h_j\rangle } & [0,1]^J}\qedhere
  \]
\end{proof}

\begin{theorem}
  Every metrisable partially ordered compact space is $\aleph_1$-copresentable
  in $\POSCH$.
\end{theorem}
\begin{proof}
  Let $X$ be a metrisable partially ordered compact space. By
  Lemma~\ref{d:lem:2}, there is an embedding
  $m \colon X\hookrightarrow [0,1]^\N$ in $\POSCH$. Moreover, with
  \[
    J=\{F\subseteq \N \mid F\text{ is finite}\},
  \]
  $(\pi_F \colon [0,1]^\N\to [0,1]^F)_{F\in J}$ is a limit cone of the
  codirected diagram
  \[
    J^\op \longrightarrow \POSCH
  \]
  sending $F$ to $[0,1]^F$ and $G\supseteq F$ to the canonical projection
  $\pi \colon [0,1]^G\to [0,1]^F$. For every $F\in J$, we consider the
  (Epi,Regular mono)-factorisation
  \[
    X\xrightarrow{\;p_F\;}X_F\xrightarrow{\;m_F\;}[0,1]^F
  \]
  of $\pi_F\cdot m \colon X\to [0,1]^F$. Then, using again Bourbaki's criterion
  (see Theorem~\ref{thm:codirected-limits-COMPHAUS}),
  \[
    (p_F \colon X\to A_F)_{F\in J}
  \]
  is a limit cone of the codirect diagram
  \[
    J^\op \longrightarrow \POSCH
  \]
  sending $F$ to $X_F$ and $G\supseteq F$ to the diagonal of the
  factorisation. By Corollary~\ref{d:cor:1}, each $X_F$ is
  $\aleph_1$-copresentable, hence also $X$ is $\aleph_1$-copresentable since $X$
  is a countable limit of $\aleph_1$-copresentable objects.
\end{proof}

\section*{Acknowledgements}

This work is financed by the ERDF -- European Regional Development Fund through
the Operational Programme for Competitiveness and Internationalisation --
COMPETE 2020 Programme and by National Funds through the Portuguese funding
agency, FCT -- Funda\c{c}\~{a}o para a Ci\^{e}ncia e a Tecnologia within project
POCI-01-0145-FEDER-016692. We also gratefully acknowledge partial financial
assistance by Portuguese funds through CIDMA (Center for Research and
Development in Mathematics and Applications), and the Portuguese Foundation for
Science and Technology (``FCT -- Funda\c{c}\~ao para a Ci\^encia e a
Tecnologia''), within the project UID/MAT/04106/2013. Finally, Renato Neves and
Pedro Nora are also supported by FCT grants SFRH/BD/52234/2013 and
SFRH/BD/95757/2013, respectively.

\newcommand{\etalchar}[1]{$^{#1}$}



\end{document}